\documentclass[12pt]{amsart}
\usepackage[left=2.5cm,right=2.5cm]{geometry}

\usepackage[T2A,T1]{fontenc}
\usepackage[utf8]{inputenc}

\usepackage{graphicx}

\usepackage{amsthm}
\usepackage{amssymb} 
\usepackage{amsmath}



\newcommand{\mylongtitle}[1]{%
	\ifodd\value{page}%
	\protect\parbox{0.9\linewidth}{#1}\hfill%
	\else%
	\hfill\protect\parbox{0.9\linewidth}{#1}%
	\fi%
}

\title[\mylongtitle{Complete characterizations of hyperbolic Coxeter groups with Sierpi\'nski curve boundary
	and with Menger curve boundary}]{
	Complete characterizations of hyperbolic Coxeter groups with Sierpi\'nski curve boundary
	and with Menger curve boundary}

\author{Daniel Danielski, Michael Kapovich and Jacek \'Swi\k{a}tkowski }

\date{}

\newtheorem{thm}{Theorem}[section]    
\newtheorem{thm+}{Theorem}[subsubsection]
\newtheorem{lem}[thm]{Lemma}          
\newtheorem{prop}[thm]{Proposition}

\theoremstyle{definition}
\newtheorem{defn}[thm]{Definition}    

\newtheorem{rems}[thm]{Remarks}
\newtheorem*{unrem}{Remark}             

\newtheorem*{rem:3.B.13}{Remark/Example/Exercise 3.B.13}

\theoremstyle{plain}
\newtheorem*{namedthm}{\namedthmname}
\newcounter{namedthm}


\theoremstyle{definition}
\newtheorem*{nameddefn}{\nameddefnname}
\newcounter{nameddefn}


\numberwithin{equation}{subsection}

\renewenvironment{proof}{\noindent{\bfseries Proof: }}{\qed}

\makeatletter
\newcommand{\double}{\mathbin{\mathpalette\make@circled\star}}
\newcommand{\make@circled}[2]{%
\ooalign{$\m@th#1\smallbigcirc{#1}$\cr\hidewidth$\m@th#1#2$\hidewidth\cr}%
}
\newcommand{\smallbigcirc}[1]{%
\vcenter{\hbox{\scalebox{0.77778}{$\m@th#1\bigcirc$}}}%
}
\makeatother

\newcommand{\Addresses}{{
	\bigskip
	\footnotesize
	\noindent
	D.~Danielski, 
	\par\noindent
	\textsc{Instytut Matematyczny, Uniwersytet Wroc\l awski,
		pl. Grunwaldzki 2, 50-384 Wroc\l aw, Poland}\par\nopagebreak
	\noindent
	\textit{E-mail address}: \texttt{danielski@math.uni.wroc.pl}
	
	\medskip
	\noindent
	M.~Kapovich, 
	\par\noindent
	\textsc{Department of Mathematics, University of California, Davis, One Shields Ave, Davis, CA 95616
	}\par\nopagebreak
	\noindent  
	\textit{E-mail address}: \texttt{kapovich@math.ucdavis.edu}
	
	\medskip
	\noindent
	J.~\'Swi\k{a}tkowski, 
	\par\noindent
	\textsc{Instytut Matematyczny, Uniwersytet Wroc\l awski,
		pl. Grunwaldzki 2, 50-384 Wroc\l aw, Poland}\par\nopagebreak
	\noindent  
	\textit{E-mail address}: \texttt{swiatkow@math.uni.wroc.pl}

}}

\begin{document}
\maketitle

\begin{abstract}
	We give complete characterizations (in terms of nerves) of those 
	word hyperbolic Coxeter groups whose 
	Gromov boundary is homeomorphic to the Sierpi\'nski curve and to the Menger curve,
	respectively.
	The justification is mostly an appropriate combination of various results from the literature.
\end{abstract}


\section{Introduction}\label{intro}

\subsection{Overview and context}
It is a classical and widely open problem to characterize those word hyperbolic groups whose
Gromov boundary is homeomorphic to a given topological space. The complete answers 
(for nonelementary hyperbolic groups) are
known only for the Cantor set (virtually free groups) and for the circle $S^1$
(cocompact Fuchsian groups). For the sphere $S^2$ the expected answer is known as Cannon's Conjecture,
and it remains open. Some partial answers are known in the restricted frameworks. 
For example, Cannon's conjecture is known to be true for Coxeter groups (we discuss 
this issue with more details in Subsection \ref{subsect1_4}). 
In this paper we deal with spaces known as the
Sierpi\'nski curve and the Menger curve, providing complete characterizations of word hyperbolic
Coxeter groups for which these spaces appear as the Gromov boundaries.

Some partial results in this direction have been presented quite recently by several authors.
For example, P.~Dani, M.~Haulmark and G.~Walsh in \cite{dhw} have shown that 
for a word hyperbolic right-angled Coxeter group $W$ whose nerve $L$ is 1-dimensional, 
$\partial W$ is homeomorphic to the Menger curve iff $L$ is {\it unseparable} 
(i.e.~connected, with no separating vertex and no separating pair of nonadjacent vertices)
and non-planar. The third author of the present paper, in \cite{js}, characterized
those word hyperbolic Coxeter groups with  Sierpi\'nski curve boundary whose nerves are
planar complexes. The first author in \cite{dd} provided
a sufficient condition for the nerve of a word hyperbolic right-angled 
Coxeter group $W$, which can be applied to nerves of arbitrary dimension, under which the Gromov boundary $\partial W$ is the Menger curve.

This paper resulted from an observation
(by the second author)
that some results of M.~Bourdon and B.~Kleiner from \cite{bk}
can be applied to obtain the complete characterizations,
as presented below.

\subsection{Results}
Before stating our main result we need to recall some terminology and notation appearing in its
statement.
The {\it nerve} of a Coxeter system $(W,S)$ is the simplicial complex $L=L(W,S)$
whose vertex set is identified with $S$ and whose simplices correspond to those subsets
$T\subset S$ for which the special subgroup $W_T$ is finite. The {\it labelled nerve} $L^\bullet$
of $(W,S)$ is the nerve $L$ in which the edges are equipped with labels in such a way that
any edge $[s,t]$ has label equal to the exponent $m_{st}$ from the standard presentation
associated to $(W,S)$ (equivalently, $m_{st}$ is the appropriate entry of the Coxeter matrix
of the system $(W,S)$). 
Obviously, the labelled nerve of a Coxeter system carries the same information as its Coxeter matrix.
Note that the labelled nerve of the direct product of
two Coxeter systems is the simplicial join of the nerves of the two factors, where the labels
at edges of the joined complexes are preserved, and the labels at all ``connecting'' edges (i.e.~edges
having endpoints in both joined complexes) are equal to 2. We call such a labelled nerve
{\it the labelled join} of the labelled nerves of the two factors.
A Coxeter system is called {\it indecomposable} if it cannot be expressed as a direct product of non-trivial Coxeter systems. Observe that a Coxeter system is indecomposable iff its labelled nerve cannot be expressed as a labelled join of two non-trivial labelled complexes.

We use the convention of speaking of topological or simplicial properties
of labelled nerves as of the properties of the corresponding underlying
unlabelled nerves.
The labelled nerve of a Coxeter system is {\it unseparable} if it is connected, has no separating simplex,
no separating pair of nonadjacent vertices,
and no separating {\it labelled suspension} (i.e.~a full subcomplex which is the labelled join of
a simplex and a doubleton). The concept of unseparability 
is useful because of the following characterization
of nonexistence of a splitting along a finite or a 2-ended subgroup in a Coxeter group,
due to Mihalik and Tschantz \cite{mt}:
{\it the group $W$ in a Coxeter system $(W,S)$ has no nontrivial splitting along a finite or
	a 2-ended subgroup iff its labelled nerve is unseparable}
(see Subsection \ref{subsect1_2} for more details).

Given a finite simplicial complex $K$ we define its {\it puncture-respecting cohomological
	dimension}, denoted as $\text{\rm pcd}(K)$, by the formula
$$
\text{\rm pcd}(K):=\max\{ n:\overline H^n(K)\ne0 \text{ \rm or }
\overline H^n(K\setminus\sigma)\ne0
\text{ \rm for some $\sigma\in \mathcal{S}(K)$} \},
$$
where $\mathcal{S}(K)$ is the family of all closed simplices of $K$.
This concept is useful for us due to its role in a formula (by M. Davis) for the virtual cohomological
dimension of a Coxeter group, see Proposition 1.3 below, and its proof.

A {\it 3-cycle} is a triangulation of the circle $S^1$ consisting of precisely 3 edges.

\smallskip
Our main result is the following.

\begin{thm}\label{thm:main}
	Let $(W,S)$ be an indecomposable Coxeter system such that $W$ is infinite word hyperbolic, 
	and let $L^\bullet$ be its labelled nerve. 
	\begin{enumerate}
		\item The Gromov boundary $\partial W$
		is homeomorphic to the Sierpi\'nski curve iff   
		$L^\bullet$ is unseparable, planar (in particular, not a triangulation of $S^2$), and not a 3-cycle. 
		
		\item The Gromov boundary $\partial W$ is homeomorphic to the Menger curve iff  
		$L^\bullet$ is unseparable,  $\text{\rm pcd}(L^\bullet)=1$, and $L^\bullet$ is not planar.
		
	\end{enumerate}

\end{thm}

\begin{rems}
	
	\begin{enumerate}
		\item Recall that $W$ is infinite iff its nerve is not a simplex.
		Recall also that word hyperbolicity of $W$ has been characterized by G.~Moussong
		(see \cite{m}, or Theorem 12.6.1 in \cite{da})
		as follows: $W$ is word hyperbolic iff it has no affine special subgroup of rank $\ge3$,
		and no special subgroup which decomposes as the direct product of two infinite
		special subgroups.

		\item
		One of the consequences of the above Moussong's characterization of word hyperbolicity
		is as follows. A word hyperbolic infinite Coxeter group decomposes (uniquely) into the direct
		product of an infinite indecomposable special subgroup
		(which is also word hyperbolic)
		and a finite special subgroup
		(possibly trivial). This allows to extend Theorem 0.1 in the obvious way to Coxeter
		systems $(W,S)$ which are not necessarily indecomposable. Namely, conditions for the nerve
		$L^\bullet$ have to be satisfied up to the labelled join with a simplex.
		
		\item The above two remarks show that Theorem 0.1 actually yields a complete characterization
		(in terms of Coxeter matrices or labelled nerves) 
		of those Coxeter systems $(W,S)$ for which $W$ is word
		hyperbolic and its Gromov boundary $\partial W$ is homeomorphic to the Sierpi\'nski curve 
		or to the Menger curve. 
		We skip the straightforward details of such characterizations.
		
	\end{enumerate}

\end{rems}

{\bf 0.3. Plan of the paper.}
In Section \ref{sect1} we collect various (rather numerous) preparatory results,
and in Section \ref{sect2} we provide the main line of the argument of the proof
of Theorem \ref{thm:main} (which is relatively short).

More precisely, here is the structure of Section \ref{sect1}.
In Subsection \ref{subsect1-1} we recall the famous topological characterizations
of the Sierpi\'nski curve and of the Menger curve, due to Whyburn \cite{wh}
and to Anderson \cite{an}, respectively. In Subsection \ref{subsect1_2} we present
a complete characterization (in terms of labelled nerves) of those
word hyperbolic Coxeter groups whose Gromov boundary is connected
and has no local cut points. As we explain, this characterization is a more
or less direct consequence of the results of Bowditch \cite{b}, Davis \cite{da2,da},
and Mihalik and Tschantz \cite{mt}. 
In Subsection \ref{subsect1_3} we present a useful formula for the topological dimension
of the Gromov boundary of a word hyperbolic Coxeter group,
which is due to Davis \cite{da2} and Bestvina and Mess \cite{bm}.
In Subsection \ref{subsect1_4} we recall a result of Bourdon and Kleiner \cite{bk}, which
confirms the Cannon's conjecture in the framework of word hyperbolic
Coxeter groups. In Subsection \ref{subsect1_6} we discuss another result,
which is implicit in the paper \cite{bk}
by Bourdon
and Kleiner, 
namely the fact that if the Gromov boundary of 
an indecomposable word hyperbolic Coxeter group is the Sierpi\'nski
curve then the nerve of the corresponding Coxeter system is a planar
simplicial complex. Since the 
arguments for
this fact provided in \cite{bk} 
are
extremely
sketchy, we include an extended exposition of 
its proof. 
In particular,
in this exposition we refer to some auxiliary result from combinatorial group
theory, which we state and prove in Subsection \ref{subsect1_5},
and for which we couldn't find an appropriate reference in the literature.

The proof of Theorem \ref{thm:main} provided in Section \ref{sect2} is split into separate parts
concerning the Menger curve and the Sierpi\'nski curve. It uses all the
preparatory results from Section \ref{sect1}.

\bigskip\noindent
{\bf Acknowledgements.} The first author was partially supported by (Polish) Narodowe Centrum Nauki, grant no  2020/37/N/ST1/01952. The third author was partially supported by (Polish) Narodowe Centrum Nauki, grant no UMO-2017/25/B/ST1/01335 .


\section{Preliminaries and preparations}\label{sect1}

In this section we collect various useful results from the literature
(or
some more or less direct
consequences of such results), and few other preparatory observations.
We will refer to all these results in our main arguments in Section
\ref{sect2}.

\subsection{Characterizations of the Sierpi\'nski curve and the Menger curve}\label{subsect1-1}

By a result of Whyburn \cite{wh}, the Sierpi\'nski curve is the unique metrizable topological space
which is compact, connected, locally connected, 1-dimensional, without local cut points
and planar. A somewhat similar result of Anderson \cite{an} characterizes the Menger curve as the
unique compact metrisable space which is connected, locally connected, 1-dimensional, has
no local cut points, and is nowhere planar (nowhere planarity means that no open
subset of the space is planar).

By referring to the above characterizations, 
the second author
and B.~Kleiner made in their
paper \cite{kk} the following observation.

\begin{prop}[M.~Kapovich and B.~Kleiner  \cite{kk}]\label{prop:characterizations}
	Let $G$ be a word hyperbolic group, and suppose that its Gromov boundary $\partial G$
	is connected, 1-dimensional, and has no local cut points. Then $\partial G$ is homeomorphic
	either to the Sierpi\'nski curve or to the Menger curve.
\end{prop}

\subsection{Connectedness and non-existence of local cut points in the Gromov boundary $\partial W$}\label{subsect1_2}

It is a well known fact that once a hyperbolic group is 1-ended
then its Gromov boundary is not only connected, but also locally connected
(see e.g.~Theorem 7.2 in \cite{kb}).
This allows to discuss existence of local cut points in the boundary. As far as this issue,
we have the following 
observation,
which probably belongs to folklore.

\bigskip
\begin{prop}\label{prop:unseparability}
	Let $(W,S)$ be a  Coxeter system, and let $L^\bullet$ be its labelled nerve. Suppose also that 
	the group $W$ is infinite and word hyperbolic. Then the Gromov boundary $\partial W$ is connected and
	has no local cut points iff $L^\bullet$ is unseparable and not a 3-cycle.
\end{prop}

\begin{proof}
	{\it Step 1.} Since connectedness of the boundary $\partial W$ is equivalent to 1-endedness of $W$,
	by Theorem 8.7.2 in \cite{da} we get that $\partial W$ is connected iff the nerve $L$
	is connected and has no separating simplex.
	
	\bigskip\noindent
	{\it Step 2.}
	By Theorem 8.7.3 in \cite{da}, a Coxeter group is 2-ended iff it decomposes as the direct product
	of its infinite dihedral special subgroup and its finite (possibly trivial) special subgroup.
	Equivalently, a Coxeter group is 2-ended iff its labelled nerve is either a doubleton or
	a labelled suspension (as defined in the introduction). 
	
	As a consequence of the above, if the group $W$ is 1-ended, 
	non-existence of a separating pair of non-adjacent vertices and of a separating 
	labelled suspension (in the labelled nerve $L^\bullet$) means 
	exactly that $W$ does not visually split (in the sense of the
	paper \cite{mt} by Mihalik and Tschantz) over a 2-ended subgroup. More precisely,
	this means that $W$ cannot be expressed as an essential free product of its two special
	subgroups, amalgamated along a 2-ended special subgroup. It follows from the
	main result of the same paper \cite{mt} that non-existence of a separating
	pair of non-adjacent vertices and of a separating
	labelled suspension in $L^\bullet$ is equivalent
	to the fact that $W$ does not split along any 2-ended subgroup.
	
	\bigskip\noindent
	{\it Step 3.}
	By a result of Bowditch \cite{b}, the Gromov boundary $\partial G$ of a 1-ended hyperbolic group $G$
	has no local cut point iff $G$ has no splitting along a 2-ended subgroup and is not a cocompact
	Fuchsian group.  By a result of Davis (see Theorem B in \cite{da2} or Theorem 10.9.2 in \cite{da}),
	a Coxeter group is
	a cocompact Fuchsian group iff its nerve is either a triangulation of $S^1$ or the group splits
	as the direct product of a special subgroup with the nerve $S^1$, and another special subgroup,
	which is finite.
	It follows from these two results, and from the conclusion of Step 2, that the Gromov boundary
	$\partial W$
	of a 1-ended word hyperbolic Coxeter group $W$ has no local cut point iff its labelled nerve
	$L^\bullet$ has no separating pair of non-adjacent vertices, no separating labelled suspension,
	and is not a 3-cycle.
	
	\bigskip\noindent
	{\it Step 4.}
	Proposition \ref{prop:unseparability} follows by combining the observations of Steps 1 and 3.
\end{proof}

\subsection{Topological dimension of the Gromov boundary $\partial W$}\label{subsect1_3}

Recall that,
given a finite simplicial complex $K$ we have defined (in the introduction) 
its {\it puncture-respecting cohomological
	dimension}, denoted as $\text{\rm pcd}(K)$, by the formula
$$
\text{\rm pcd}(K):=\max\{ n:\overline H^n(K)\ne0 \text{ \rm or }
\overline H^n(K\setminus\sigma)\ne0
\text{ \rm for some $\sigma\in \mathcal{S}(K)$} \},
$$
where $\mathcal{S}(K)$ is the family of all closed simplices of $K$.
The role of this concept for our considerations in this paper comes from the following
observation. 

\bigskip

\begin{prop}\label{prop:top_dim}
	Let $(W,S)$ be a Coxeter system, and let $L$ be its nerve. Suppose also that 
	the group $W$ is word hyperbolic. Then
	$$
	\dim\partial W=\text{\rm pcd}(L).
	$$
\end{prop}

\begin{proof}
	Denote by $\text{\rm vcd}(W)$ the virtual cohomological dimension of $W$. It follows from
	results of Mike Davis that
	$\text{\rm vcd}(W) = \text{\rm pcd}(L)+1$
	(see Corollary 8.5.5 in \cite{da}). On the other hand, by the result of M.~Bestvina and G.~Mess
	\cite{bm}, we have $\text{\rm vcd}(W) =\dim\partial W+1$, hence the proposition.
\end{proof}

\subsection{Cannon's conjecture for Coxeter groups}\label{subsect1_4}

The following result has been proved using quite advanced methods by M.~Bourdon
and B.~Kleiner in \cite{bk}, and its short proof as presented below (indicated by M. Davis)
has been also outlined in the same paper.
We include this short proof for completeness (since our statement, being
convenient for our applications,
is not identical to that in \cite{bk}), and for reader's convenience.

\bigskip
\begin{prop}\label{prop:cannon}
	Let $(W,S)$ be an indecomposable Coxeter system, and let $L$ be its nerve. Suppose also that 
	the group $W$ is word hyperbolic. Then the following conditions are equivalent:
	
	\begin{enumerate}
		
		\item $\partial W\cong S^2$,
		
		\item $L$ is a triangulation of $S^2$,
		
		\item $W$ acts properly discontinuously and cocompactly, by isometries, as a reflection group,
		on the hyperbolic space $\mathbb{H}^3$.
		
	\end{enumerate}
	
\end{prop}

\begin{proof}
	We justify the implications $1.\Rightarrow 2.\Rightarrow 3.\Rightarrow 1.$
	
	\medskip\noindent
	{\it Proof of \enskip$1.\Rightarrow 2.$} By result of M.~Bestvina and G.~Mess (Corollary 1.3(c)
	in \cite{bm}), if $\partial W\cong S^2$ then $W$ is a virtual Poincar\'e duality group
	of dimension 3. By result of M.~Davis (Theorem 10.9.2 in \cite{da}), the nerve $L$ 
	is then a triangulation of $S^2$ (here we use the assumption of indecomposability).
	
	\medskip\noindent
	{\it Proof of \enskip$2.\Rightarrow 3.$} This implication follows by applying Andreev's theorem
	(see \cite{a}, or Theorem 6.10.2 in \cite{da}) to the dual polyhedron of the triangulation. 
	
	\medskip\noindent
	{\it Proof of \enskip$3.\Rightarrow 1.$} By the assumptions on $W$ in condition 3, we obviously
	have $\partial W=\partial\mathbb{H}^3$, and the implication follows
	from the fact that
	$\partial\mathbb{H}^3\cong S^2$.
\end{proof}

\bigskip
For the later arguments of this paper we only need the implication $1.\Rightarrow 2$.

\subsection{An observation from combinatorial group theory}\label{subsect1_5}

Let $\Gamma$ be an arbitrary group and $H_i$ for $1\leq i\leq n$ be a collection of its (not necessarily pairwise distinct) subgroups. In this subsection we describe two group operations associated to this data, and discuss
the relationship between the groups obtained by these operations. 
This
observation (Lemma \ref{lem:doubles} below) will be useful in the argument in Subsection \ref{subsect1_6}.

In the next definition we describe the first  of the two operations, 
which 
the second author
and B.~Kleiner call 
the {\it double} of $\Gamma$ with respect to the 
tuple $(H_i)$
(see \cite{kk}).

\begin{defn}\label{defn:Defn1.5.1}
	Given a group $\Gamma$ and a finite 
	tuple
	of its subgroups 
	$(H_i)$,
	the {\it double} $\Gamma\double\Gamma$ is  the fundamental group
	$\pi_1{\mathcal G}$ of the graph of groups $\mathcal{G}$ described as follows. The underlying graph 
	of $\mathcal G$ consists of two vertices
	$v$ and $v'$ and $n$ edges $e_1,\dots,e_n$ each of which has both $v$ and $v'$ as its
	endpoints. The vertex groups at $v$ and $v'$ are both identified with $\Gamma$ while the edge
	group at any edge $e_i$ is identified with $H_i$. The structure homomorphisms are all
	taken to be the inclusions.
\end{defn}

Let $\Gamma=\langle S|R\rangle$, and let $\Gamma'=\langle S'|R'\rangle$ be a second copy of $\Gamma$ (given by the same presentation).
Denote by $\mathcal{W}_{H_i}$ the set of words over $S\cup S^{-1}$ that represent elements of the subgroup $H_i$ and for a word $w$ over $S\cup S^{-1}$ let $w'$ be the word over $S'\cup S'^{-1}$ obtained from $w$ by substituting each letter with its counterpart from $S'\cup S'^{-1}$.
Note that  (e.g.~by Definition 7.3 in \cite{dicdun}), the double $\Gamma\double\Gamma$ can be
also described as follows. 
Consider an auxilliary group 
$P=P(\Gamma,(H_i))$
given by the presentation
$$\langle S\sqcup S'\sqcup\{u_i:1\leq i\leq n\}|R\cup R'\cup\{h_i u_i=u_ih_i':1\leq i\leq n,h_i\in \mathcal{W}_{H_i}\} \rangle.$$
Then $\Gamma\double\Gamma$ is a subgroup of $P$
consisting of all elements $p$ such that there exists an expression $p=w_0u_{i_1}w_1u_{i_2}^{-1}w_2\ldots w_{2m-1}u_{i_{2m}}^{-1}w_{2m}$ for some $m\ge0, 1\leq i_k\leq n$ and words $w_k$ over $S\cup S^{-1}$ and $S'\cup S'^{-1}$ for even and odd $k$ respectively.

\medskip
The second of the group operations is given in the following.

\begin{defn}\label{defn:Defn1.5.2}
	Given a group $\Gamma=\langle S|R \rangle$ and a finite 
	tuple
	of its subgroups 
	$(H_i)$,
	the {\it mirror double} 
	$\widetilde{\Gamma}$ of the group $\Gamma$ with respect to the 
	tuple $(H_i)$,
	is the group given by the presentation 
	
	\begin{align*}
		\widetilde{\Gamma}:=&\langle S\sqcup\{s_i:1\leq i\leq n\}|   \nonumber  \\
		&R\cup \{s_i^2=1:1\leq i\leq n\} \cup\{h_is_i=s_ih_i:1\leq i\leq n,h_i\in \mathcal{W}_{H_i}\} \rangle.
		\nonumber  
	\end{align*}

\end{defn} 

Observe that the mirror double is (up to isomorphism) independent of the presentation of $\Gamma$ used in the definition above.

\begin{lem}\label{lem:doubles}
	For each group $\Gamma$ and any finite 
	tuple
	of its subgroups 
	$(H_i)$
	the double $\Gamma\double\Gamma$ is isomorphic to an index 2 subgroup of the mirror double $\widetilde{\Gamma}$.
\end{lem}

\begin{unrem}
	The concepts of a double $\Gamma\double\Gamma$ and a mirror double 
	$\widetilde{\Gamma}$ are well known e.g.~in the context of compact
	hyperbolic manifolds, $M$, with nonempty totally geodesic boundary
	$\partial M$. If we take $\Gamma=\pi_1M$, and if subgroups
	$H_i<\Gamma$ correspond to the fundamental groups of the boundary
	components, the double $\Gamma\double\Gamma$ is the fundamental group
	of the double $DM$ of the manifold $M$ along $\partial M$.
	In the same situation, the mirror double $\widetilde{\Gamma}$
	corresponds to the fundamental group of the orbifold $\mathcal{O}_M$
	with the underlying space $M$, in which the local groups at the boundary
	are the groups of order 2 representing geometrically local reflections. 
	Since the double $DM$ is obviously a degree 2 covering of the orbifold
	$\mathcal{O}_M$ (in the orbifold sense), the assertion of Lemma
	\ref{lem:doubles} is obvious in this situation. The full statement
	of Lemma \ref{lem:doubles} is just a group theoretic extension
	of that observation (which could be also given a geometrical sense).
	
\end{unrem}

\begin{proof}
	Consider the homomorphism $\rho\colon P\to\widetilde{\Gamma}$ given by 
	$\rho(s)=\rho(s')=s$ for each $s\in S$, and $ \rho(u_i)=s_i$ for each $1\leq i\leq n$. 
	Consider also the homomorphism 
	$\sigma\colon\widetilde{\Gamma}\to\mathbb{Z}_2$ given by  $\sigma(s)=0$ for 
	$s\in S$, and $\sigma(s_i)=1$ for $1\leq i\leq n$. It suffices to show that $\rho$ restricts to an isomorphism between $\Gamma\double\Gamma$ and $\ker\sigma$, which is an index 2 subgroup of $\widetilde{\Gamma}$. It is easy to check that $\rho(\Gamma\double\Gamma)=\ker\sigma$, so it remains to show that $\rho|_{\Gamma\double\Gamma}$ is injective. To this end we introduce the following lift function $\ell\colon \ker\sigma\to\Gamma\double\Gamma$. 
	For an element $\xi\in\ker\sigma$, and for its any expression by a word of the form $w_0s_{i_1}^{\epsilon_1}w_1s_{i_2}^{\epsilon_2}\cdot\ldots\cdot w_{2m-1}s_{i_{2m}}^{\epsilon_{2m}}w_{2m}$ for some (possibly empty) words $w_i$ over the alphabet $S\cup S^{-1}$, $\epsilon_j\in\{-1,1\}$ and for $1\leq i_j\leq n$, put $\ell(\xi):=w_0u_{i_1}w_1'u_{i_2}^{-1}\cdot\ldots\cdot w_{2m-1}'u_{i_{2m}}^{-1}w_{2m}$.
	The map $\ell$ is well defined, since it is easy to check that for each word 
	$$
	U=w_0s_{i_1}^{\epsilon_1}w_1s_{i_2}^{\epsilon_2}\cdot\ldots\cdot w_{2m-1}s_{i_{2m}}^{\epsilon_{2m}}w_{2m},
	$$ 
	and for each elementary operation consisting of inserting at an arbitrary place in $U$ (or deleting) a subword of the form $a^{-1}a$ for some letter $a$, or a relator (in $\widetilde{\Gamma}$) or inverse of such, resulting in the word $\hat{U}$, the words representing $\ell(U)$ and $\ell(\hat{U})$ in the definition of $\ell$ differ by an analogous elementary operation (in $P$). Moreover,
	since we then clearly have that $\ell\circ \rho|_{\Gamma\double\Gamma}=\mathrm{id}_{\Gamma\double\Gamma}$, we conclude that $\rho|_{\Gamma\double\Gamma}$ is injective,
	hence the lemma.	
\end{proof}

\subsection{Planarity of nerves}\label{subsect1_6}

We 
recall the following rather easy observation from the paper \cite{js} written by the
third
author of the present paper. 

\begin{lem}[J.~\'Swiątkowski, Lemma 1.3 in \cite{js}]\label{lem:planar}
	If the nerve $L$ of a word hyperbolic Coxeter group $W$ is a planar complex then
	the Gromov boundary $\partial W$ is a planar topological space.
\end{lem}

The converse implication is not true in general \cite{dhw},
but it does hold in an important special case. This is the contents of
the next result which appears implicitly
as Corollary 7.5 in \cite{bk}. The  proof given below is an expansion
of a rather sketchy argument provided in \cite{bk}.

\begin{prop}\label{prop:planar}
	Let $(W,S)$ be an indecomposable Coxeter system such that the group $W$ is word hyperbolic.
	If the Gromov boundary $\partial W$ is homeomorphic
	to the Sierpi\'nski curve then the nerve $L$ of the system $(W,S)$ is a planar
	simplicial complex.
\end{prop}

\begin{proof}
	We will embed the group $W$, as a special subgroup, in some larger indecomposable
	and word hyperbolic Coxeter group $\widetilde W$ such that $\partial \widetilde W\cong S^2$.
	The assertion will follow then from the implication $1.\Rightarrow 2.$ in Proposition 
	\ref{prop:cannon}.
	
	We start by recalling some facts established in the paper \cite{kk} by 
	the second author 
	and B.~Kleiner. 
	First, the Sierpi\'nski curve contains the family of topologically 
	well distinguished pairwise disjoint subsets
	homeomorphic to $S^1$, called {\it peripheral circles}. 
	Moreover, in its action on $\partial W$ the group $W$ 
	maps peripheral circles to peripheral circles. 
	A setwise stabilizer of each peripheral circle in $\partial W$, called a {\it peripheral subgroup} of $W$,
	is a quasi-convex subgroup of $W$ for which
	the circle is its limit set, and consequently each such stabilizer is a cocompact Fuchsian group.
	The action of $W$ on the family of peripheral circles in $\partial W$ 
	has finitely many orbits, and thus we have
	finitely many conjugacy classes of peripheral subgroups in $W$.
	
	\bigskip\noindent
	{\bf Claim.} 
	{\it Each peripheral subgroup of $W$ is a conjugate of some special subgroup of $W$.}
	
	\medskip
	To prove this claim we need some terminology and notation as in
	Section 5.1 in \cite{bk}. For a generator $s\in S$,
	the {\it wall} $M_s$ is the set of setwise $s$-stabilized open edges of $\mathrm{Cay}(W,S)$ (the Cayley graph of $W$ with respect to the set of generators $S$). Then $\mathrm{Cay}(W,S)\setminus M_s$ consists of two connected components $H_-(M_s)$ and $H_+(M_s)$.
	For a generator $s\in S$ and for an arbitrary element $g\in W$ we consider the reflection $r:=gsg^{-1}$, its wall $M_r:=gM_s$ and components $H_-(M_r)$ and $H_+(M_r)$ of $\mathrm{Cay}(W,S)\setminus M_r$.
	The components are closed and convex subsets of $\mathrm{Cay}(W,S)$ and $\partial H_-(M_r)\cup\partial H_+(M_r)=\partial W$, $\partial H_-(M_r)\cap\partial H_+(M_r)=\partial M_r$ and $r$ pointwise stabilizes $\partial M_r$.

	\bigskip\noindent
	{\bf Proof of Claim:} In view of Definition 5.4 and Theorem 5.5 in \cite{bk} it suffices to show that for each peripheral circle $F$ and each reflection $r$ such that $\partial H_-(M_r)\cap F$ and $\partial H_+(M_r)\cap F$ are non-empty, it holds that $F$ is setwise stabilized by $r$. Since $(\partial H_-(M_r)\cap F)\cup(\partial H_+(M_r)\cap F)=\partial W\cap F=F$, by connectedness of $F\cong S^1$, we have that $\emptyset\neq (H_-(\partial M_r)\cap F)\cap(H_+(\partial M_r)\cap F)=\partial M_r\cap F$. Since $\partial M_r$ is pointwise stabilized by $r$, $rF\cap F\neq\emptyset$, and, finally, $rF=F$
	by the fact that each element of $W$ maps peripheral circles to peripheral circles.
	
	\bigskip
	Coming back to the proof of Proposition \ref{prop:planar},
	denote by $H_i:1\le i\le n$ a set of representatives of the conjugacy classes of peripheral
	subgroups of $W$ consisting of special subgroups of $W$. 
	For each $1\le i\le n$, denote by $L_i$ the nerve of $H_i$, and view it as a subcomplex of
	the nerve $L$ of $W$.
	We will discuss below the double $W\double W$ and the mirror double $\widetilde{W}$ of $W$ with respect to the 
	tuple $(H_i)$
	(see Subsection \ref{subsect1_5}).
	As it is shown in \cite{kk}, the double $W\double W$ is a hyperbolic group and its Gromov boundary is homeomorphic to $S^2$.
	Observe also that the mirror double $\widetilde W$ is (isomorphic to) a Coxeter group with  nerve
	$\widetilde L$ obtained from the nerve $L$ of $W$ by adding a simplicial cone over each of the 
	subcomplexes $L_i$.
	Moreover, since each $H_i$ is a proper special subgroup of $W$, indecomposability of $W$
	implies indecomposability of $\widetilde W$.  By Lemma \ref{lem:doubles}, the group
	$\widetilde W$ contains $W\double W$
	as a subgroup of index 2, and hence it is also word hyperbolic and its Gromov boundary is homeomorphic to $S^2$.
	By Proposition \ref{prop:cannon},  $\widetilde L$ is then a triangulation
	of $S^2$. Since $L$ is clearly a proper subcomplex of $\widetilde L$, it is necessarily planar,
	which completes the proof of Proposition \ref{prop:planar}.
\end{proof}

\section{Proof of the main theorem}\label{sect2}

\subsection{Sierpi\'nski curve boundary}

In this rather short subsection we prove part 1 of Theorem \ref{thm:main}.

\medskip\noindent
{\bf Proof of the implication $\Rightarrow$.}
Suppose that $\partial W$ is homeomorphic to the Sierpi\'nski curve.
Then, in view of the fact that the Sierpi\'nski curve is connected and has no local cut points, 
it follows from Proposition \ref{prop:unseparability} that
$L^\bullet$ is unseparable and not a 3-cycle.
Moreover, by Proposition \ref{prop:planar}, $L$ is then a planar simplicial complex,
which completes the proof.

\medskip\noindent
{\bf Proof of the implication $\Leftarrow$.}
As any Gromov boundary of a hyperbolic group, $\partial W$ is a compact metrisable space.
Since $L$ is planar, it follows from Lemma \ref{lem:planar} that $\partial W$ is a planar space.
Since $L^\bullet$ is unseparable and not a 3-cycle, it follows from Proposition
\ref{prop:unseparability} that $\partial W$ is connected, locally connected, and has no local cut point.
Finally, it is not hard to see that since $L$ is planar, connected, has no separating simplex,
and does not coincide with a single simplex,
its puncture-respecting cohomological dimension $\text{\rm pcd}(L)$ is equal to 1.
Consequently, due to Proposition \ref{prop:top_dim}, $\partial W$ has topological dimension 1. 
Thus, by Whyburn's characterization recalled in Subsection \ref{subsect1-1},
$\partial W$ is homeomorphic to the Sierpi\'nski curve, as required.

\subsection{Menger curve boundary}

We now pass to the proof of part 2 of Theorem \ref{thm:main}.

\medskip\noindent
{\bf Proof of the implication $\Rightarrow$.}
Suppose that $\partial W$ is homeomorphic to the Menger curve.
Then, in view of the fact that the Menger curve is connected and has no local cut points, 
it follows from Proposition \ref{prop:unseparability} that
$L^\bullet$ is unseparable. Since the Menger curve has topological dimension 1,
it follows from Proposition \ref{prop:top_dim} that $\text{\rm pcd}(L)=1$.
Since the Menger curve is not planar, it follows from Lemma \ref{lem:planar}
that $L$ is also not planar, and this completes the proof of the first implication.

\medskip\noindent
{\bf Proof of the implication $\Leftarrow$.}
The boundary $\partial W$ is obviously a compact metrisable space.
Since $L^\bullet$ is not planar, not a 3-cycle, 
and since $L^\bullet$ is unseparable, it follows from Proposition
\ref{prop:unseparability} that $\partial W$ is connected, locally connected, and has no local cut point.
Since $\text{\rm pcd}(L)=1$, it follows that $\partial W$ has topological dimension 1.
In view of the above properties, it follows from Proposition \ref{prop:characterizations}
that $\partial W$ is homeomorphic either to the Sierpi\'nski curve or to the Menger curve.
However, since $L$ is not planar, it follows from Proposition \ref{prop:planar} that
$\partial W$ cannot be homeomorphic to the Sierpi\'nski curve. Consequently,
it must be homeomorphic to the Menger curve, as required.

\Addresses

\end{document}